\documentclass{article}
\usepackage{geometry,amsmath,amsthm,amssymb}
\usepackage{authblk}

\usepackage[colorlinks=true,linkcolor=blue,citecolor=blue,pdfpagelabels=false]{hyperref}

\usepackage{thmtools}
\theoremstyle{plain}
  \declaretheorem[numberwithin=section]{theorem}
  \declaretheorem[numberlike=theorem]{corollary}
  
  \declaretheorem[numberlike=theorem]{lemma}

\theoremstyle{definition}
  \declaretheorem[numberlike=theorem]{definition}
  \declaretheorem[numberlike=theorem]{example}
  \declaretheorem[numberlike=theorem]{remark}
\newenvironment{acknowledgements}{\bigskip\textbf{Acknowledgements.}}{}

\newcommand{\assign}{:=}

\begin{document}

\title{Gaussian binomial coefficients with negative arguments}

\author{Sam Formichella}
\author{Armin Straub\thanks{\texttt{straub@southalabama.edu}}}
\affil{Department of Mathematics and Statistics\\University of South Alabama}

\date{February 7, 2018}

\maketitle

\begin{abstract}
  Loeb showed that a natural extension of the usual binomial coefficient to
  negative (integer) entries continues to satisfy many of the fundamental
  properties. In particular, he gave a uniform binomial theorem as well as a
  combinatorial interpretation in terms of choosing subsets of sets with a
  negative number of elements. We show that all of this can be extended to the
  case of Gaussian binomial coefficients. Moreover, we demonstrate that
  several of the well-known arithmetic properties of binomial coefficients
  also hold in the case of negative entries. In particular, we show that
  Lucas' Theorem on binomial coefficients modulo $p$ not only extends
  naturally to the case of negative entries, but even to the Gaussian case.
\end{abstract}

\section{Introduction}

Occasionally, the binomial coefficient $\binom{n}{k}$, with integer entries
$n$ and $k$, is considered to be zero when $k < 0$ (see
Remark~\ref{rk:conventions}). However, as observed by Loeb \cite{loeb-neg},
there exists an alternative extension of the binomial coefficients to negative
arguments, which is arguably more natural for many combinatorial or number
theoretic applications. The $q$-binomial coefficients $\binom{n}{k}_q$ (often
also referred to as Gaussian polynomials) are a polynomial generalization of
the binomial coefficients that occur naturally in varied contexts, including
combinatorics, number theory, representation theory and mathematical physics.
For instance, if $q$ is a prime power, then they count the number of
$k$-dimensional subspaces of an $n$-dimensional vector space over the finite
field $\mathbb{F}_q$. We refer to the book \cite{kc-q} for a very nice
introduction to the $q$-calculus. Yet, surprisingly, $q$-binomial coefficients
with general integer entries have, to the best of our knowledge, not been
introduced in the literature (Loeb \cite{loeb-neg} does briefly discuss
$q$-binomial coefficients but only in the case $k \geq 0$). The goal of
this paper is to fill this gap, and to demonstrate that these generalized
$q$-binomial coefficients are natural, by showing that they satisfy many of
the fundamental combinatorial and arithmetic properties of the usual binomial
coefficients. In particular, we extend Loeb's interesting combinatorial
interpretation \cite{loeb-neg} in terms of sets with negative numbers of
elements. On the arithmetic side, we prove that Lucas' theorem can be
uniformly generalized to both binomial coefficients and $q$-binomial
coefficients with negative entries.

In the context of $q$-series, it is common to introduce the $q$-binomial
coefficient, for $n, k \geq 0$, as the quotient
\begin{equation}
  \binom{n}{k}_q = \frac{(q ; q)_n}{(q ; q)_k (q ; q)_{n - k}},
  \label{eq:qbin:def:poch}
\end{equation}
where $(a ; q)_n$ denotes the $q$-Pochhammer symbol
\begin{equation*}
  (a ; q)_n \assign \prod_{j = 0}^{n - 1} (1 - a q^j), \quad n \geq 0.
\end{equation*}
In particular, $(a ; q)_0 = 1$. It is not difficult to see that
\eqref{eq:qbin:def:poch} reduces to the usual binomial coefficient in the
limit $q \rightarrow 1$. In order to extend \eqref{eq:qbin:def:poch} to the
case of negative integers $n$ and $k$, we observe that the simple relation
\begin{equation*}
  (a ; q)_n = \frac{(a ; q)_{\infty}}{(a q^n ; q)_{\infty}}
\end{equation*}
can be used to extend the $q$-Pochhammer symbol to the case when $n < 0$. That
is, if $n < 0$, it is common to define
\begin{equation*}
  (a ; q)_n \assign \prod_{j = 1}^{| n |} \frac{1}{1 - a q^{- j}} .
\end{equation*}
Note that $(q ; q)_n = \infty$ whenever $n < 0$, so that
\eqref{eq:qbin:def:poch} does not immediately extend to the case when $n$ or
$k$ are negative. We therefore make the following definition, which clearly
reduces to \eqref{eq:qbin:def:poch} when $n, k \geq 0$.

\begin{definition}
  For all integers $n$ and $k$,
  \begin{equation}
    \binom{n}{k}_q \assign \lim_{a \rightarrow q} \frac{(a ; q)_n}{(a ; q)_k
    (a ; q)_{n - k}} . \label{eq:qbin:def:pochlim}
  \end{equation}
\end{definition}

Though not immediately obvious from \eqref{eq:qbin:def:pochlim} when $n$ or
$k$ are negative, these generalized $q$-binomial coefficients are Laurent
polynomials in $q$ with integer coefficients. In particular, upon setting $q =
1$, we always obtain integers.

\begin{example}
  \label{eg:qbin:neg}
  \begin{equation*}
    \binom{- 3}{- 5}_q = \lim_{a \rightarrow q} \frac{(a ; q)_{- 3}}{(a ;
     q)_{- 5} (a ; q)_2} = \lim_{a \rightarrow q} \frac{\left(1 -
     \frac{a}{q^4} \right) \left(1 - \frac{a}{q^5} \right)}{(1 - a) (1 - a
     q)} = \frac{(1 + q^2) (1 + q + q^2)}{q^7}
  \end{equation*}
\end{example}

In Section~\ref{sec:pascal}, we observe that, for integers $n$ and $k$, the
$q$-binomial coefficients are also characterized by the Pascal relation
\begin{equation}
  \binom{n}{k}_q = \binom{n - 1}{k - 1}_q + q^k \binom{n - 1}{k}_q,
  \label{eq:qbin:pascal:intro}
\end{equation}
provided that $(n, k) \neq (0, 0)$ (this exceptional case excludes itself
naturally in the proof of Lemma~\ref{lem:qbin:pascal}), together with the
initial conditions
\begin{equation*}
  \binom{n}{0}_q = \binom{n}{n}_q = 1.
\end{equation*}
In the case $q = 1$, this extension of Pascal's rule to negative parameters
was observed by Loeb \cite[Proposition~4.4]{loeb-neg}.

Among the other basic properties of the generalized $q$-binomial coefficient
are the following. All of these are well-known in the classical case $n, k
\geq 0$. That they extend uniformly to all integers $n$ and $k$ (though,
as illustrated by \eqref{eq:qbin:pascal:intro} and item \ref{i:qbin:neg}, some
care has to be applied when generalizing certain properties) serves as a first
indication that the generalized $q$-binomial coefficients are natural objects.
For \ref{i:qbin:neg}, the sign function $\operatorname{sgn} (k)$ is defined to be $1$
if $k \geq 0$, and $- 1$ if $k < 0$.

\begin{lemma}
  \label{lem:qbin:basic}For all integers $n$ and $k$,
  \begin{enumerate}
    \item \label{i:qbin:qinv}$\binom{n}{k}_q = q^{k (n - k)}
    \binom{n}{k}_{q^{- 1}}$,
    
    \item \label{i:qbin:sym}$\binom{n}{k}_q = \binom{n}{n - k}_q$,
    
    \item \label{i:qbin:neg}$\binom{n}{k}_q = (- 1)^k \operatorname{sgn} (k)
    q^{\frac{1}{2} k (2 n - k + 1)} \binom{k - n - 1}{k}_q$,
    
    \item \label{i:qbin:hyp:nk}$\binom{n}{k}_q = \frac{1 - q^n}{1 - q^k}
    \binom{n - 1}{k - 1}_q$ if $k \neq 0$.
  \end{enumerate}
\end{lemma}

Properties \ref{i:qbin:sym} and \ref{i:qbin:hyp:nk} follow directly from the
definition \eqref{eq:qbin:def:pochlim}, while property \ref{i:qbin:qinv} is
readily deduced from \eqref{eq:qbin:pascal:intro} combined with
\ref{i:qbin:sym}. In the classical case $n, k \geq 0$, property
\ref{i:qbin:qinv} reflects the fact that the $q$-binomial coefficient is a
self-reciprocal polynomial in $q$ of degree $k (n - k)$. In contrast to that
and as illustrated in Example~\ref{eg:qbin:neg}, the $q$-binomial coefficients
with negative entries are Laurent polynomials, whose degrees are recorded in
Corollary~\ref{cor:qbin:deg}.

The reflection rule \ref{i:qbin:neg} is the subject of
Section~\ref{sec:reflect} and is proved in Theorem~\ref{thm:qbin:neg}. Rule
\ref{i:qbin:neg} reduced to the case $q = 1$ is the main object in
\cite{sprugnoli-binom}, where Sprugnoli observed the necessity of including
the sign function when extending the binomial coefficient to negative entries.
Sprugnoli further realized that the basic symmetry \ref{i:qbin:sym} and the
negation rule \ref{i:qbin:neg} act on binomial coefficients as a group of
transformations isomorphic to the symmetric group on three letters. In
Section~\ref{sec:reflect}, we observe that the same is true for $q$-binomial
coefficients.

Note that property \ref{i:qbin:hyp:nk}, when combined with \ref{i:qbin:sym},
implies that, for $n \neq k$,
\begin{equation*}
  \binom{n}{k}_q = \frac{1 - q^n}{1 - q^{n - k}} \binom{n - 1}{k}_q .
\end{equation*}
In particular, the $q$-binomial coefficient is a $q$-hypergeometric term.

\begin{example}
  It follows from Lemma~{{\ref{lem:qbin:basic}}}\ref{i:qbin:neg} that, for all
  integers $k$,
  \begin{equation*}
    \binom{- 1}{k}_q = (- 1)^k \operatorname{sgn} (k)  \frac{1}{q^{k (k + 1) / 2}} .
  \end{equation*}
\end{example}

In Section~\ref{sec:combinatorial}, we review the remarkable and beautiful
observation of Loeb \cite{loeb-neg} that the combinatorial interpretation of
binomial coefficients as counting subsets can be naturally extended to the
case of negative entries. We then prove that this interpretation can be
generalized to $q$-binomial coefficients. Theorem~\ref{thm:qbin:subset}, our
main result of that section, is a precise version of the following.

\begin{theorem}
  \label{thm:qbin:subset:intro}For all integers $n$ and $k$,
  \begin{equation*}
    \binom{n}{k}_q = \pm \sum_Y q^{\sigma (Y) - k (k - 1) / 2},
  \end{equation*}
  where the sum is over all $k$-element subsets $Y$ of the $n$-element set
  $X_n$.
\end{theorem}

The notion of sets (and subsets) with a negative number of elements, as well
as the definitions of $\sigma$ and $X_n$, are deferred to
Section~\ref{sec:combinatorial}. In the previously known classical case $n, k
\geq 0$, the sign in that formula is positive, $X_n = \{ 0, 1, 2, \ldots,
n - 1 \}$, and $\sigma (Y)$ is the sum of the elements of $Y$. As an
application of Theorem~\ref{thm:qbin:subset:intro}, we demonstrate at the end
of Section~\ref{sec:combinatorial} how to deduce from it generalized versions
of the Chu-Vandermonde identity as well as the (commutative) $q$-binomial
theorem.

In Section~\ref{sec:binomialtheorem}, we discuss the binomial theorem, which
interprets the binomial coefficients as coefficients in the expansion of $(x +
y)^n$. Loeb showed that, by also considering expansions in inverse powers of
$x$, one can extend this interpretation to the case of binomial coefficients
with negative entries. Once more, we are able to show that the generalized
$q$-binomial coefficients share this property in a uniform fashion.

\begin{theorem}
  \label{thm:binomialtheorem:loeb:q:intro}Suppose that $y x = q x y$. Then,
  for all integers $n, k$,
  \begin{equation*}
    \binom{n}{k}_q = \{ x^k y^{n - k} \} (x + y)^n .
  \end{equation*}
\end{theorem}

Here, the operator $\{ x^k y^{n - k} \}$, which is defined in
Section~\ref{sec:binomialtheorem}, extracts the coefficient of $x^k y^{n - k}$
in the appropriate expansion of what follows.

A famous theorem of Lucas \cite{lucas78} states that, if $p$ is a prime,
then
\begin{equation*}
  \binom{n}{k} \equiv \binom{n_0}{k_0} \binom{n_1}{k_1} \cdots
   \binom{n_d}{k_d} \quad (\operatorname{mod} p),
\end{equation*}
where $n_i$ and $k_i$ are the $p$-adic digits of the nonnegative integers $n$
and $k$, respectively. In Section~\ref{sec:lucas}, we show that this
congruence in fact holds for all integers $n$ and $k$. In fact, in
Section~\ref{sec:lucas:q}, we prove that generalized Lucas congruences
uniformly hold for $q$-binomial coefficients.

\begin{theorem}
  \label{thm:qbin:lucas:intro}Let $m \geq 2$ be an integer. Then, for all
  integers $n$ and $k$,
  \begin{equation*}
    \binom{n}{k}_q \equiv \binom{n_0}{k_0}_q \binom{n'}{k'} \quad (\operatorname{mod}
     \Phi_m (q)),
  \end{equation*}
  where $n = n_0 + n' m$ and $k = k_0 + k' m$ with $n_0, k_0 \in \{ 0, 1,
  \ldots, m - 1 \}$.
\end{theorem}

Here, $\Phi_m (q)$ is $m$th cyclotomic polynomial. The classical special case
$n, k \geq 0$ of this result has been obtained by Olive
\cite{olive-pow} and D\'esarm\'enien \cite{desarmenien-q}.

We conclude this introduction with some comments on alternative approaches to
and conventions for binomial coefficients with negative entries. In
particular, we remark on the current state of computer algebra systems and how
it could benefit from the generalized $q$-binomial coefficients introduced in
this paper.

\begin{remark}
  \label{rk:bin:gamma}Using the gamma function, binomial coefficients can be
  introduced as
  \begin{equation}
    \binom{n}{k} \assign \frac{\Gamma (n + 1)}{\Gamma (k + 1) \Gamma (n - k +
    1)} \label{eq:binom:gamma}
  \end{equation}
  for all complex $n$ and $k$ such that $n, k \not\in \{ - 1, - 2, \ldots \}$.
  This definition, however, does not immediately lend itself to the case of
  negative integers; though the structure of poles (and lack of zeroes) of the
  underlying gamma function is well understood, the binomial function
  \eqref{eq:binom:gamma} has a subtle structure when viewed as a function of
  two variables. For a study of this function, as well as a historical account
  on binomials, we refer to \cite{fowler-bin}. A natural way to extend
  \eqref{eq:binom:gamma} to negative integers is to set
  \begin{equation}
    \binom{n}{k} \assign \lim_{\varepsilon \rightarrow 0} \frac{\Gamma (n + 1
    + \varepsilon)}{\Gamma (k + 1 + \varepsilon) \Gamma (n - k + 1 +
    \varepsilon)}, \label{eq:binom:gamma:eps}
  \end{equation}
  where $n$ and $k$ are now allowed to take any complex values. This is in
  fact the definition that Loeb \cite{loeb-neg} and Sprugnoli
  \cite{sprugnoli-binom} adopt. (That the $q$-binomial coefficients we
  introduce in \eqref{eq:qbin:def:pochlim} reduce to
  \eqref{eq:binom:gamma:eps} when $q = 1$ can be seen, for instance, from
  observing that the Pascal relation \eqref{eq:qbin:pascal:intro} reduces to
  the relation established by Loeb for \eqref{eq:binom:gamma:eps}.)
\end{remark}

\begin{remark}
  \label{rk:conventions}Other conventions for binomial coefficients with
  negative integer entries exist and have their merit. Most prominently, if,
  for instance, one insists that Pascal's relation
  \eqref{eq:qbin:pascal:intro} should hold for all integers $n$ and $k$, then
  the resulting version of the binomial coefficients is zero when $k < 0$. On
  the other hand, as illustrated by the results in \cite{loeb-neg} and this
  paper, it is reasonable and preferable for many purposes to extend the
  classical binomial coefficients (as well as its polynomial counterpart) to
  negative arguments as done here.
  
  As an unfortunate consequence, both conventions are implemented in current
  computer algebra systems, which can be a source of confusion. For instance,
  SageMath currently (as of version $8.0$) uses the convention that all
  binomial coefficients with $k < 0$ are evaluated as zero. On the other hand,
  recent versions of Mathematica (at least version $9$ and higher) and Maple
  (at least version $18$ and higher) evaluate binomial coefficients with
  negative entries in the way advertised in \cite{loeb-neg} and here.
  
  In version 7, Mathematica introduced the \texttt{QBinomial[n,k,q]}
  function; however, as of version 11, this function evaluates the
  $q$-binomial coefficient as zero whenever $k < 0$. Similarly, Maple
  implements these coefficients as \texttt{QBinomial(n,k,q)}, but, as of
  version 18, choosing $k < 0$ results in a division-by-zero error. We hope
  that this paper helps to adjust these inconsistencies with the classical
  case $q = 1$ by offering a natural extension of the $q$-binomial coefficient
  for negative entries.
\end{remark}

\section{Characterization via a $q$-Pascal relation}\label{sec:pascal}

The generalization of the binomial coefficients to negative entries by Loeb
satisfies Pascal's rule
\begin{equation}
  \binom{n}{k} = \binom{n - 1}{k - 1} + \binom{n - 1}{k} \label{eq:bin:pascal}
\end{equation}
for all integers $n$ and $k$ that are not both zero
\cite[Proposition~4.4]{loeb-neg}. In this brief section, we show that the
$q$-binomial coefficients (with arbitrary integer entries), defined in
\eqref{eq:qbin:def:pochlim}, are also characterized by a $q$-analog of the
Pascal rule. It is well-known that this is true for the familiar $q$-binomial
coefficients when $n, k \geq 0$ (see, for instance,
\cite[Proposition~6.1]{kc-q}).

\begin{lemma}
  \label{lem:qbin:pascal}For integers $n$ and $k$, the $q$-binomial
  coefficients are characterized by
  \begin{equation}
    \binom{n}{k}_q = \binom{n - 1}{k - 1}_q + q^k \binom{n - 1}{k}_q,
    \label{eq:qbin:pascal}
  \end{equation}
  provided that $(n, k) \neq (0, 0)$, together with the initial conditions
  \begin{equation*}
    \binom{n}{0}_q = \binom{n}{n}_q = 1.
  \end{equation*}
\end{lemma}

Observe that $\binom{0}{0}_q = 1$, while the corresponding right-hand side of
\eqref{eq:qbin:pascal} is $\binom{- 1}{- 1}_q + q^0 \binom{- 1}{0}_q = 2 \neq
1$, illustrating the need to exclude the case $(n, k) = (0, 0)$. It should
also be noted that the initial conditions are natural but not minimal: the
case $\binom{n}{0}_q$ with $n \leq - 2$ is redundant (but consistent).

\begin{proof}
  We note that the relation \eqref{eq:qbin:pascal} and the initial conditions
  indeed suffice to deduce values for each $q$-binomial coefficient. It
  therefore only remains to show that \eqref{eq:qbin:pascal} holds for the
  $q$-binomial coefficient as defined in \eqref{eq:qbin:def:pochlim}. For the
  purpose of this proof, let us write
  \begin{equation*}
    \binom{n}{k}_{a, q} \assign \frac{(a ; q)_n}{(a ; q)_k (a ; q)_{n - k}},
  \end{equation*}
  and observe that, for all integers $n$ and $k$,
  \begin{equation*}
    \binom{n - 1}{k}_{a, q} = \frac{1 - a q^{n - k - 1}}{1 - a q^{n - 1}}
     \binom{n}{k}_{a, q}
  \end{equation*}
  as well as
  \begin{equation*}
    \binom{n - 1}{k - 1}_{a, q} = \frac{1 - a q^{k - 1}}{1 - a q^{n - 1}}
     \binom{n}{k}_{a, q} .
  \end{equation*}
  It then follows that
  \begin{equation}
    \binom{n}{k}_{a, q} = \binom{n - 1}{k - 1}_{a, q} + a q^{k - 1} \frac{1 -
    q^{n - k}}{1 - a q^{n - k - 1}} \binom{n - 1}{k}_{a, q}
    \label{eq:qbin:pascal:a}
  \end{equation}
  for all integers $n$ and $k$. If $n \neq k$, then
  \begin{equation*}
    \lim_{a \rightarrow q} \left[ a q^{k - 1} \frac{1 - q^{n - k}}{1 - a q^{n
     - k - 1}} \right] = q^k,
  \end{equation*}
  so that \eqref{eq:qbin:pascal} follows for these cases. On the other hand,
  if $n = k$, then $\binom{n - 1}{k}_q = 0$, provided that $(n, k) \neq (0,
  0)$, so that \eqref{eq:qbin:pascal} also holds in the remaining cases.
\end{proof}

\begin{remark}
  Applying Pascal's relation \eqref{eq:qbin:pascal} to the right-hand side of
  Lemma~\ref{lem:qbin:basic}\ref{i:qbin:sym}, followed by applying the
  symmetry Lemma~\ref{lem:qbin:basic}\ref{i:qbin:sym} to each $q$-binomial
  coefficient, we find that Pascal's relation \eqref{eq:qbin:pascal} is
  equivalent to the alternative form
  \begin{equation}
    \binom{n}{k}_q = q^{n - k} \binom{n - 1}{k - 1}_q + \binom{n - 1}{k}_q .
    \label{eq:qbin:pascal2}
  \end{equation}
\end{remark}

\section{Reflection formulas}\label{sec:reflect}

In \cite{sprugnoli-binom}, Sprugnoli, likely unaware of the earlier work of
Loeb \cite{loeb-neg}, introduces binomial coefficients with negative entries
via the gamma function (see Remark~\ref{rk:bin:gamma}). Sprugnoli then
observes that the familiar negation rule
\begin{equation*}
  \binom{n}{k} = (- 1)^k \binom{k - n - 1}{k}
\end{equation*}
as stated, for instance, in \cite[Section~1.2.6]{knuth-acp1}, does not
continue to hold when $k$ is allowed to be negative. Instead, he shows that,
for all integers $n$ and $k$,
\begin{equation}
  \binom{n}{k} = (- 1)^k \operatorname{sgn} (k) \binom{k - n - 1}{k},
  \label{eq:bin:neg}
\end{equation}
where $\operatorname{sgn} (k) = 1$ for $k \geq 0$ and $\operatorname{sgn} (k) = - 1$ for
$k < 0$. We generalize this result to the $q$-binomial coefficients. Observe
that the result of Sprugnoli \cite{sprugnoli-binom} is immediately obtained
as the special case $q = 1$.

\begin{theorem}
  \label{thm:qbin:neg}For all integers $n$ and $k$,
  \begin{equation}
    \binom{n}{k}_q = (- 1)^k \operatorname{sgn} (k) q^{\frac{1}{2} k (2 n - k + 1)}
    \binom{k - n - 1}{k}_q . \label{eq:qbin:neg}
  \end{equation}
\end{theorem}

\begin{proof}
  Let us begin by observing that, for all integers $n$ and $k$,
  \begin{equation}
    (a ; q)_n (a q^n ; q)_k = (a ; q)_{n + k} . \label{eq:poch:split}
  \end{equation}
  Further, for all integers $n$,
  \begin{equation}
    (a ; q)_n = (- a)^n q^{n (n - 1) / 2} (q^{- n + 1} / a ; q)_n .
    \label{eq:poch:inva}
  \end{equation}
  Applying \eqref{eq:poch:split} and then \eqref{eq:poch:inva}, we find that
  \begin{equation*}
    \frac{(a ; q)_n}{(a ; q)_{n - k}} = \frac{1}{(a q^n ; q)_{- k}} =
     \frac{(- a)^k q^{\frac{1}{2} k (2 n - k - 1)}}{(q^{k - n + 1} / a ; q)_{-
     k}} .
  \end{equation*}
  By another application of \eqref{eq:poch:split},
  \begin{equation*}
    \frac{1}{(q^{k - n + 1} / a ; q)_{- k}} = \frac{(1 / a ; q)_{k - n +
     1}}{(1 / a ; q)_{- n + 1}} = \frac{(q^2 / a ; q)_{k - n - 1}}{(q^2 / a ;
     q)_{- n - 1}},
  \end{equation*}
  where, for the second equality, we used the basic relation $(a ; q)_n = (1 -
  a) (a q ; q)_{n - 1}$ twice for each Pochhammer symbol. Combined, we thus
  have
  \begin{equation*}
    \frac{(a ; q)_n}{(a ; q)_{n - k}} = (- a)^k q^{\frac{1}{2} k (2 n - k -
     1)} \frac{(q^2 / a ; q)_{k - n - 1}}{(q^2 / a ; q)_{- n - 1}}
  \end{equation*}
  for all integers $n$ and $k$. Suppose we have already shown that, for any
  integer $n$,
  \begin{equation}
    \lim_{a \rightarrow q} \frac{(q^2 / a ; q)_n}{(a ; q)_n} = \operatorname{sgn} (n)
    . \label{eq:poch:sgn}
  \end{equation}
  Then,
  \begin{eqnarray*}
    \binom{n}{k}_q & = & \lim_{a \rightarrow q} \frac{(a ; q)_n}{(a ; q)_k (a
    ; q)_{n - k}}\\
    & = & \lim_{a \rightarrow q} (- a)^k q^{\frac{1}{2} k (2 n - k - 1)}
    \frac{(q^2 / a ; q)_{k - n - 1}}{(a ; q)_k (q^2 / a ; q)_{- n - 1}}\\
    & = & \operatorname{sgn} (k - n - 1) \operatorname{sgn} (- n - 1) \lim_{a \rightarrow q}
    (- a)^k q^{\frac{1}{2} k (2 n - k - 1)} \frac{(a ; q)_{k - n - 1}}{(a ;
    q)_k (a ; q)_{- n - 1}}\\
    & = & (- 1)^k \operatorname{sgn} (k) q^{\frac{1}{2} k (2 n - k + 1)} \binom{k - n
    - 1}{k}_q .
  \end{eqnarray*}
  For the final equality, we used that
  \begin{equation*}
    \operatorname{sgn} (k - n - 1) \operatorname{sgn} (- n - 1) = \operatorname{sgn} (k)
  \end{equation*}
  whenever the involved $q$-binomial coefficients are different from zero (for
  more details on this argument, see \cite[Theorem~2.2]{sprugnoli-binom}).
  
  It remains to show \eqref{eq:poch:sgn}. The limit clearly is $1$ if $n
  \geq 0$. On the other hand, if $n < 0$, then
  \begin{eqnarray*}
    \lim_{a \rightarrow q} \frac{(q^2 / a ; q)_n}{(a ; q)_n} & = & \lim_{a
    \rightarrow q} \frac{\left(1 - \frac{a}{q} \right) \left(1 -
    \frac{a}{q^2} \right) \cdots \left(1 - \frac{a}{q^n} \right)}{\left(1 -
    \frac{q}{a} \right) \left(1 - \frac{1}{a} \right) \cdots \left(1 -
    \frac{1}{a q^{n - 2}} \right)}\\
    & = & \lim_{a \rightarrow q} \frac{\left(1 - \frac{a}{q} \right)}{\left(1 - \frac{q}{a} \right)} = - 1,
  \end{eqnarray*}
  as claimed.
\end{proof}

It was observed in \cite[Theorem~3.2]{sprugnoli-binom} that the basic
symmetry (Lemma~\ref{lem:qbin:basic}\ref{i:qbin:sym}) and the negation rule
\eqref{eq:qbin:neg} act on (formal) binomial coefficients as a group of
transformations isomorphic to the symmetric group on three letters. The same
is true for $q$-binomial coefficients. Since the argument is identical, we
only record the resulting six forms for the $q$-binomial coefficients.

\begin{corollary}
  \label{cor:qbin:forms}For all integers $n$ and $k$,
  \begin{eqnarray*}
    \binom{n}{k}_q & = & \binom{n}{n - k}_q\\
    & = & (- 1)^{n - k} \operatorname{sgn} (n - k) q^{\frac{1}{2} (n (n + 1) - k (k +
    1))} \binom{- k - 1}{n - k}_q\\
    & = & (- 1)^{n - k} \operatorname{sgn} (n - k) q^{\frac{1}{2} (n (n + 1) - k (k +
    1))} \binom{- k - 1}{- n - 1}_q\\
    & = & (- 1)^k \operatorname{sgn} (k) q^{\frac{1}{2} k (2 n - k + 1)} \binom{k - n
    - 1}{- n - 1}_q\\
    & = & (- 1)^k \operatorname{sgn} (k) q^{\frac{1}{2} k (2 n - k + 1)} \binom{k - n
    - 1}{k}_q .
  \end{eqnarray*}
\end{corollary}

\begin{proof}
  These equalities follow from alternately applying the basic symmetry from
  Lemma~\ref{lem:qbin:basic}\ref{i:qbin:sym} and the negation rule
  \eqref{eq:qbin:neg}. Moreover, for the fourth equality, we use that
  \begin{equation*}
    - \operatorname{sgn} (n - k) \operatorname{sgn} (- n - 1) = \operatorname{sgn} (k)
  \end{equation*}
  whenever the involved $q$-binomial coefficients are different from zero
  (again, see \cite[Theorem~2.2]{sprugnoli-binom} for more details on this
  argument).
\end{proof}

It follows directly from the definition \eqref{eq:qbin:def:pochlim} that the
$q$-binomial coefficient $\binom{n}{k}_q$ is zero if $k > n \geq 0$ or if
$n \geq 0 > k$. The third equality in Corollary~\ref{cor:qbin:forms} then
makes it plainly visible that the $q$-binomial coefficient also vanishes if $0
> k > n$. Moreover, we can read off from Corollary~\ref{cor:qbin:forms} that
the $q$-binomial coefficient is nonzero otherwise; that is, it is nonzero
precisely in the three regions $0 \leq k \leq n$ (the classical
case), $n < 0 \leq k$ and $k \leq n < 0$. More precisely, we have
the following, of which the first statement is, of course, well-known (see,
for instance, \cite[Corollary~6.1]{kc-q}).

\begin{corollary}
  \label{cor:qbin:deg}\leavevmode

  \begin{enumerate}
    \item If $0 \leq k \leq n$, then $\binom{n}{k}_q$ is a
    polynomial of degree $k (n - k)$.
    
    \item If $n < 0 \leq k$, then $\binom{n}{k}_q$ is $q^{\frac{1}{2} k
    (2 n - k + 1)}$ times a polynomial of degree $k (- n - 1)$.
    
    \item If $k \leq n < 0$, then $\binom{n}{k}_q$ is $q^{\frac{1}{2} (n
    (n + 1) - k (k + 1))}$ times a polynomial of degree $(- n - 1) (n - k)$.
  \end{enumerate}
  In each case, the polynomials are self-reciprocal and have integer
  coefficients.
\end{corollary}

Observe that Corollary~\ref{cor:qbin:forms} together with the defining product
\eqref{eq:qbin:def:poch}, spelled out as
\begin{equation*}
  \binom{n}{k}_q = \frac{(1 - q^{k + 1}) (1 - q^{k + 2}) \cdots (1 - q^n)}{(1
   - q) (1 - q^2) \cdots (1 - q^{n - k})}
\end{equation*}
and valid when $0 \leq k \leq n$, provides explicit product formulas
for all choices of $n$ and $k$. Indeed, the three regions in which the
binomial coefficients are nonzero are $0 \leq k \leq n$, $n < 0
\leq k$ and $k \leq n < 0$, and these three are permuted by the
transformations in Corollary~\ref{cor:qbin:forms}.

\section{Combinatorial interpretation}\label{sec:combinatorial}

For integers $n, k \geq 0$, the binomial coefficient $\binom{n}{k}$
counts the number of $k$-element subsets of a set with $n$ elements. It is a
remarkable and beautiful observation of Loeb \cite{loeb-neg} that this
interpretation (up to an overall sign) can be extended to all integers $n$ and
$k$ by a natural notion of sets with a negative number of elements. In this
section, after briefly reviewing Loeb's result, we generalize this
combinatorial interpretation to the case of $q$-binomial coefficients.

Let $U$ be a collection of elements (the ``universe''). A set $X$ with
elements from $U$ can be thought of as a map $M_X : U \rightarrow \{ 0, 1 \}$
with the understanding that $u \in X$ if and only if $M_X (u) = 1$. Similarly,
a multiset $X$ can be thought of as a map $M_X : U \rightarrow \{ 0, 1, 2,
\ldots \}$, in which case $M_X (u)$ is the multiplicity of an element $u$. In
this spirit, Loeb introduces a {\emph{hybrid set}} $X$ as a map $M_X : U
\rightarrow \mathbb{Z}$. We will denote hybrid sets in the form $\{ \ldots |
\ldots \}$, where elements with a positive multiplicity are listed before the
bar, and elements with a negative multiplicity after the bar.

\begin{example}
  The hybrid set $\{ 1, 1, 4|2, 3, 3 \}$ contains the elements $1, 2, 3, 4$
  with multiplicities $2, - 1, - 2, 1$, respectively.
\end{example}

A hybrid set $Y$ is a subset of a hybrid set $X$, if one can repeatedly remove
elements from $X$ (here, removing means decreasing by one the multiplicity of
an element with nonzero multiplicity) and thus obtain $Y$ or have removed $Y$.
We refer to \cite{loeb-neg} for a more formal definition and further
discussion, including a proof that this notion of being a subset is a
well-defined partial order (but not a lattice).

\begin{example}
  From the hybrid set $\{ 1, 1, 4|2, 3, 3 \}$ we can remove the element $4$ to
  obtain $\{ 1, 1|2, 3, 3 \}$ (at which point, we cannot remove $4$ again). We
  can further remove $2$ twice to obtain $\{ 1, 1|2, 2, 2, 3, 3 \}$.
  Consequently, $\{ 4| \}$ and $\{ 1, 1|2, 3, 3 \}$ as well as $\{ 2, 2, 4|
  \}$ and $\{ 1, 1|2, 2, 2, 3, 3 \}$ are subsets of $\{ 1, 1, 4|2, 3, 3 \}$.
\end{example}

Following \cite{loeb-neg}, a {\emph{new set}} is a hybrid set such that
either all multiplicities are $0$ or $1$ (a ``positive set'') or all
multiplicities are $0$ or $- 1$ (a ``negative set'').

\begin{theorem}[\cite{loeb-neg}]
  \label{thm:bin:subset}For all integers $n$ and $k$, the number of
  $k$-element subsets of an $n$-element new set is $\left| \binom{n}{k}
  \right|$.
\end{theorem}

\begin{example}
  \label{eg:bin:subsets}Consider the new set $\{ | - 1, - 2, - 3 \}$ with $-
  3$ elements (the reason for choosing the elements to be negative numbers
  will become apparent when we revisit this example in
  Example~\ref{eg:qbin:subsets}). Its $2$-element subsets are
  \begin{equation*}
    \{ - 1, - 1| \}, \quad \{ - 1, - 2| \}, \quad \{ - 1, - 3| \}, \quad \{ -
     2, - 2| \}, \quad \{ - 2, - 3| \}, \quad \{ - 3, - 3| \},
  \end{equation*}
  so that $\left| \binom{- 3}{2} \right| = 6$. On the other hand, its $-
  4$-element subsets are
  \begin{equation*}
    \{ | - 1, - 1, - 2, - 3 \}, \quad \{ | - 1, - 2, - 2, - 3 \}, \quad \{ |
     - 1, - 2, - 3, - 3 \},
  \end{equation*}
  so that $\left| \binom{- 3}{- 4} \right| = 3$.
\end{example}

Let $X_n$ denote the standard new set with $n$ elements, by which we mean $X_n
= \{ 0, 1, \ldots, n - 1| \}$, if $n \geq 0$, and $X_n = \{ | - 1, - 2,
\ldots, n \}$, if $n < 0$. For a hybrid set $Y \subseteq X_n$ with
multiplicity function $M_Y$, we write
\begin{equation*}
  \sigma (Y) = \sum_{y \in Y} M_Y (y) y.
\end{equation*}
Note that, if $Y$ is a classic set, then $\sigma (Y)$ is just the sum of its
elements. With this setup, we prove the following uniform generalization of
\cite[Theorem~5.2]{loeb-neg}, which is well-known in the case that $n, k
\geq 0$ (see, for instance, \cite[Theorem~6.1]{kc-q}).

\begin{theorem}
  \label{thm:qbin:subset}For all integers $n$ and $k$,
  \begin{equation}
    \binom{n}{k}_q = \varepsilon \sum_Y q^{\sigma (Y) - k (k - 1) / 2}, \quad
    \varepsilon = \pm 1, \label{eq:qbin:subset}
  \end{equation}
  where the sum is over all $k$-element subsets $Y$ of the $n$-element set
  $X_n$. If $0 \leq k \leq n$, then $\varepsilon = 1$. If $n < 0
  \leq k$, then $\varepsilon = (- 1)^k$. If $k \leq n < 0$, then
  $\varepsilon = (- 1)^{n - k}$.
\end{theorem}

\begin{proof}
  The case $n, k \geq 0$ is well-known. A proof can be found, for
  instance, in \cite[Theorem~6.1]{kc-q}. On the other hand, if $n \geq
  0 > k$, then both sides vanish.
  
  Let us consider the case $n < 0 \leq k$. It follows from the reflection
  formula \eqref{eq:qbin:neg} that \eqref{eq:qbin:subset} is equivalent to the
  (arguably cleaner, but less uniform because restricted to $n < 0 \leq
  k$) identity
  \begin{equation}
    \binom{k - n - 1}{k}_q = \sum_{Y \in C (n, k)} q^{\sigma (Y)},
    \label{eq:qbin:subset:2}
  \end{equation}
  where $C (n, k)$ is the collection of $k$-element subsets of the $n$-element
  set $X_n^+ = \{ |0, 1, 2, \ldots, | n | - 1 \}$ (note that a natural
  bijection $X_n \rightarrow X_n^+$ is given by $x \mapsto | n | + x$).
  
  Fix $n, k$ and suppose that \eqref{eq:qbin:subset:2} holds whenever $n$ and
  $k$ are replaced with $n'$ and $k'$ such that $n < n' < 0$ or $n = n' < 0
  \leq k' < k$. Then,
  \begin{eqnarray*}
    \sum_{Y \in C (n, k)} q^{\sigma (Y)} & = &
    \sum_{\substack{
      Y \in C (n, k)\\
      - n - 1 \not\in Y
    }} q^{\sigma (Y)} + \sum_{\substack{
      Y \in C (n, k)\\
      - n - 1 \in Y
    }} q^{\sigma (Y)}\\
    & = & \sum_{Y \in C (n + 1, k)} q^{\sigma (Y)} + \sum_{Y \in C (n, k -
    1)} q^{\sigma (Y) - n - 1}\\
    & = & \binom{k - n - 2}{k}_q + q^{- n - 1} \binom{k - n - 2}{k - 1}_q\\
    & = & \binom{k - n - 1}{k}_q,
  \end{eqnarray*}
  where the last equality follows from Pascal's relation in the form
  \eqref{eq:qbin:pascal2}. Since \eqref{eq:qbin:subset:2} holds trivially if
  $n = - 1$ or if $k = 0$, it therefore follows by induction that
  \eqref{eq:qbin:subset:2} is true whenever $n < 0 \leq k$.
  
  Finally, consider the case $n, k < 0$. It is clear that both sides vanish
  unless $k \leq n < 0$. By the third equality in
  Corollary~\ref{cor:qbin:forms},
  \begin{equation*}
    \binom{n}{k}_q = (- 1)^{n - k} q^{\frac{1}{2} (n (n + 1) - k (k + 1))}
     \binom{- k - 1}{- n - 1}_q,
  \end{equation*}
  so that \eqref{eq:qbin:subset} becomes equivalent to
  \begin{equation}
    \binom{- k - 1}{- n - 1}_q = \sum_{Y \in D (n, k)} q^{\sigma (Y) + k - n
    (n + 1) / 2}, \label{eq:qbin:subset:3}
  \end{equation}
  where $D (n, k)$ is the collection of $k$-element subsets $Y$ of the
  $n$-element set $X_n = \{ | - 1, - 2, \ldots, n \}$. If $n = - 1$, then
  \eqref{eq:qbin:subset:3} holds because the only contribution comes from $Y =
  \{ | - 1, - 1, \ldots, - 1 \}$, with $M_Y (- 1) = | k |$ and $\sigma (Y) = -
  k$. If, on the other hand, $k = - 1$, then \eqref{eq:qbin:subset:3} holds as
  well because a contributing $Y$ only exists if $n = - 1$. Fix $n, k < - 1$
  and suppose that \eqref{eq:qbin:subset:3} holds whenever $n$ and $k$ are
  replaced with $n'$ and $k'$ such that $k < k' < 0$ and $n \leq n' < 0$.
  Then the right-hand side of \eqref{eq:qbin:subset:3} equals
  \begin{equation*}
    \sum_{\substack{
       Y \in D (n, k)\\
       M_Y (n) = - 1
     }} q^{\sigma (Y) + k - n (n + 1) / 2} +
     \sum_{\substack{
       Y \in D (n, k)\\
       M_Y (n) < - 1
     }} q^{\sigma (Y) + k - n (n + 1) / 2} .
  \end{equation*}
  We now remove the element $n$ from $Y$ (once) and, to make up for that,
  replace $\sigma (Y)$ with $\sigma (Y) - n$. Proceeding thus, we see that the
  right-hand side of \eqref{eq:qbin:subset:3} equals
  \begin{eqnarray*}
    &  & \sum_{Y \in D (n + 1, k + 1)} q^{\sigma (Y) + k + 1 - (n + 1) (n +
    2) / 2} + q^{- n - 1} \sum_{Y \in D (n, k + 1)} q^{\sigma (Y) + k + 1 - n
    (n + 1) / 2}\\
    & = & \binom{- k - 2}{- n - 2}_q + q^{- n - 1} \binom{- k - 2}{- n - 1}_q
    = \binom{- k - 1}{- n - 1}_q,
  \end{eqnarray*}
  with the final equality following from Pascal's relation
  \eqref{eq:qbin:pascal}. We conclude, by induction, that
  \eqref{eq:qbin:subset:3} is true for all $n, k < 0$.
\end{proof}

\begin{remark}
  The number of possibilities to choose $k$ elements from a set of $n$
  elements with replacement is
  \begin{equation*}
    \binom{k + n - 1}{k} = \binom{k + n - 1}{n - 1} .
  \end{equation*}
  The usual ``trick'' to arrive at this count is to encode each choice of $k$
  elements by lining them up in some order with elements of the same kind
  separated by dividers (since there are $n$ kinds of elements, we need $n -
  1$ dividers). The $n - 1$ positions of the dividers among all $k + n - 1$
  positions then encode a choice of $k$ elements. Formula
  \eqref{eq:qbin:subset:2} is a $q$-analog of this combinatorial fact.
\end{remark}

\begin{example}
  \label{eg:qbin:subsets}Let us revisit and refine
  Example~\ref{eg:bin:subsets}, which concerns subsets of $X_{- 3} = \{ | - 1,
  - 2, - 3 \}$. Letting $k = 2$, the $2$-element subsets have element-sums
  $\sigma (\{ - 1, - 1| \}) = - 2$, $\sigma (\{ - 1, - 2| \}) = - 3$, $\sigma
  (\{ - 1, - 3| \}) = - 4$, $\sigma (\{ - 2, - 2| \}) = - 4$, $\sigma (\{ - 2,
  - 3| \}) = - 5$, $\sigma (\{ - 3, - 3| \}) = - 6$. Subtracting $k (k - 1) /
  2 = 1$ from these sums to obtain the weight, we find
  \begin{equation*}
    \binom{- 3}{2}_q = q^{- 3} + q^{- 4} + 2 q^{- 5} + q^{- 6} + q^{- 7} .
  \end{equation*}
  Next, let us consider the case $k = - 4$. The $- 4$-element subsets have
  element-sums
  \begin{equation*}
    \sigma (\{ | - 1, - 1, - 2, - 3 \}) = 7, \quad \sigma (\{ | - 1, - 2, -
     2, - 3 \}) = 8, \quad \sigma (\{ | - 1, - 2, - 3, - 3 \}) = 9.
  \end{equation*}
  Subtracting $k (k - 1) / 2 = 10$ from these sums and noting that $(- 1)^{n -
  k} = - 1$, we conclude that
  \begin{equation*}
    \binom{- 3}{- 4}_q = - (q^{- 3} + q^{- 2} + q^{- 1}) .
  \end{equation*}
\end{example}

In the remainder of this section, we consider two applications of
Theorem~\ref{thm:qbin:subset}. The first of these is the following extension
of the classical Chu-Vandermonde identity.

\begin{lemma}
  For all integers $n, m$ and $k$, with $k \geq 0$,
  \begin{equation}
    \sum_{j = 0}^k q^{(k - j) (n - j)} \binom{n}{j}_q \binom{m}{k - j}_q =
    \binom{n + m}{k}_q . \label{eq:chuvandermonde:pos}
  \end{equation}
\end{lemma}

\begin{proof}
  Throughout this proof, if $Y$ is a $k$-element set, write $\tau (Y) = \sigma
  (Y) - k (k - 1) / 2$.
  
  Suppose $n, m \geq 0$. Let $Y_1$ be a $j$-element subset of $X_n$, and
  $Y_2$ a $(k - j)$-element subset of $X_m$. Let $Y_2' = \{ y + n : y \in Y_2
  \}$, so that $Y = Y_1 \cup Y_2'$ is a $k$-element subset of $X_{n + m}$.
  Then, since
  \begin{equation*}
    \sigma (Y) = \sigma (Y_1) + \sigma (Y_2') = \sigma (Y_1) + \sigma (Y_2) +
     (k - j) n,
  \end{equation*}
  we have
  \begin{equation*}
    \tau (Y) = \tau (Y_1) + \tau (Y_2) + (k - j) (n - j) .
  \end{equation*}
  Then this follows from Theorem~\ref{thm:qbin:subset} because
  \begin{equation*}
    \binom{j}{2} + \binom{k - j}{2} - \binom{k}{2} + (k - j) n = (k - j) (n -
     j) .
  \end{equation*}
  
\end{proof}

Similarly, one can deduce from Theorem~\ref{thm:qbin:subset} the following
version for the case when $k$ is a negative integer. It trivially also holds
if $n, m \geq 0$, but the identity does not generally hold in the case
when $n$ and $m$ have mixed signs.

\begin{lemma}
  For all negative integers $n, m$ and $k$,
  \begin{equation*}
    \sum_{j \in \{ - 1, - 2, \ldots, k + 1 \}} q^{(k - j) (n - j)}
     \binom{n}{j}_q \binom{m}{k - j}_q = \binom{n + m}{k}_q .
  \end{equation*}
\end{lemma}

As another application of the combinatorial characterization in
Theorem~\ref{thm:qbin:subset}, we readily obtain the following identity. In
the case $n \geq 0$, this identity is well-known and often referred to as
the (commutative version of the) $q$-binomial theorem (in which case the sum
only extends over $k = 0, 1, \ldots, n$). We will discuss the noncommutative
$q$-binomial theorem in the next section.

\begin{theorem}
  For all integers $n$,
  \begin{equation*}
    (- x ; q)_n = \sum_{k \geq 0} q^{k (k - 1) / 2} \binom{n}{k}_q x^k .
  \end{equation*}
\end{theorem}

\begin{proof}
  Suppose that $n \geq 0$, so that
  \begin{equation}
    (- x ; q)_n = (1 + x) (1 + x q) \cdots (1 + x q^{n - 1}) .
    \label{eq:binthm:poch:pos}
  \end{equation}
  Let, as before $X_n = \{ 0, 1, \ldots, n - 1| \}$. To each subset $Y
  \subseteq X_n$ we associate the product of the terms $x q^y$ with $y \in Y$
  in the expansion of \eqref{eq:binthm:poch:pos}. This results in
  \begin{equation*}
    (- x ; q)_n = \sum_{Y \subseteq X_n} q^{\sigma (Y)} x^{| Y |},
  \end{equation*}
  which, by Theorem~\ref{thm:qbin:subset}, reduces to the claimed sum.
  
  Next, let us consider the case $n < 0$. Then, $X_n = \{ | - 1, - 2, \ldots,
  n \}$ and
  \begin{equation*}
    (x ; q)_n = \prod_{j = 1}^{| n |} \frac{1}{1 - x q^{- j}} = \prod_{j =
     1}^{| n |} \sum_{m \geq 0} x^m q^{- j m} .
  \end{equation*}
  Similar to the previous case, terms of the expansion of this product are in
  natural correspondence with (hybrid) subsets $Y \subseteq X_n$. Namely, to
  $Y$ we associate the product of the terms $x^m q^{y m}$ where $y \in Y$ and
  $m = M_Y (y)$ is the multiplicity of $y$. Therefore,
  \begin{equation*}
    (- x ; q)_n = \sum_{Y \subseteq X_n} (- 1)^{| Y |} q^{\sigma (Y)} x^{| Y
     |},
  \end{equation*}
  and the claim again follows directly from Theorem~\ref{thm:qbin:subset}
  (note that $\varepsilon = (- 1)^k$ in the present case).
\end{proof}

\section{The binomial theorem}\label{sec:binomialtheorem}

When introducing binomial coefficients with negative entries, Loeb
\cite{loeb-neg} also provided an extension of the binomial theorem
\begin{equation*}
  (x + y)^n = \sum_{k = 0}^n \binom{n}{k} x^k y^{n - k},
\end{equation*}
the namesake of the binomial coefficients, to the case when $n$ and $k$ may be
negative integers. In this section, we show that this extension can also be
generalized to the case of $q$-binomial coefficients.

Suppose that $f (x)$ is a function with Laurent expansions
\begin{equation}
  f (x) = \sum_{k \geq - N} a_k x^k, \hspace{3em} f (x) = \sum_{k
  \geq - N} b_{- k} x^{- k} \label{eq:f:exp:0:infty}
\end{equation}
around $x = 0$ and $x = \infty$, respectively. Let us extract coefficients of
these expansions by writing
\begin{equation*}
  \{ x^k \} f (x) \assign \left\{ \begin{array}{ll}
     a_k, & \text{if $k \geq 0$,}\\
     b_k, & \text{if $k < 0$.}
   \end{array} \right.
\end{equation*}
Loosely speaking, $\{ x^k \} f (x)$ is the coefficient of $x^k$ in the
appropriate Laurent expansion of $f (x)$.

\begin{theorem}[\cite{loeb-neg}]
  \label{thm:binomialtheorem:loeb}For all integers $n$ and $k$,
  \begin{equation*}
    \binom{n}{k} = \{ x^k \} (1 + x)^n .
  \end{equation*}
\end{theorem}

\begin{example}
  As $x \rightarrow \infty$,
  \begin{equation*}
    (1 + x)^{- 3} = x^{- 3} - 3 x^{- 4} + 6 x^{- 5} + O (x^{- 6}),
  \end{equation*}
  so that, for instance,
  \begin{equation*}
    \binom{- 3}{- 5} = 6.
  \end{equation*}
\end{example}

It is well-known (see, for instance, \cite[Theorem~5.1]{kc-q}) that, if $x$
and $y$ are noncommuting variables such that $y x = q x y$, then the
$q$-binomial coefficients arise from the expansion of $(x + y)^n$.

\begin{theorem}
  \label{thm:binomialtheorem:q}Let $n \geq 0$. If $y x = q x y$, then
  \begin{equation}
    (x + y)^n = \sum_{k = 0}^n \binom{n}{k}_q x^k y^{n - k} .
    \label{eq:binomialtheorem:q:pos}
  \end{equation}
\end{theorem}

Our next result shows that the restriction to $n \geq 0$ is not
necessary. In fact, we prove the following result, which extends both the
noncommutative $q$-binomial Theorem~\ref{thm:binomialtheorem:q} and Loeb's
Theorem~\ref{thm:binomialtheorem:loeb}. In analogy with the classical case, we
consider expansions of $f_n (x, y) = (x + y)^n$ in the two $q$-commuting
variables $x, y$. As before, we can expand $f_n (x, y)$ in two different ways,
that is,
\begin{equation*}
  f_n (x, y) = \sum_{k \geq 0} a_k x^k y^{n - k}, \hspace{3em} f_n (x,
   y) = \sum_{k \geq n} b_{- k} x^{- k} y^{n + k} .
\end{equation*}
Again, we extract coefficients of these expansions by writing
\begin{equation*}
  \{ x^k y^{n - k} \} f_n (x, y) \assign \left\{ \begin{array}{ll}
     a_k, & \text{if $k \geq 0$,}\\
     b_k, & \text{if $k < 0$.}
   \end{array} \right.
\end{equation*}
\begin{theorem}
  \label{thm:binomialtheorem:loeb:q}Suppose that $y x = q x y$. Then, for all
  integers $n$ and $k$,
  \begin{equation*}
    \binom{n}{k}_q = \{ x^k y^{n - k} \} (x + y)^n .
  \end{equation*}
\end{theorem}

\begin{proof}
  Using the geometric series,
  \begin{equation*}
    (x + y)^{- 1} = y^{- 1} (x y^{- 1} + 1)^{- 1} = y^{- 1} \sum_{k \geq
     0} (- 1)^k (x y^{- 1})^k .
  \end{equation*}
  and, applying the $q$-commutativity,
  \begin{equation*}
    (x + y)^{- 1} = \sum_{k \geq 0} (- 1)^k q^{- k (k + 1) / 2} x^k y^{-
     k - 1} = \sum_{k \geq 0} \binom{- 1}{k}_q x^k y^{- 1 - k} .
  \end{equation*}
  (Consequently, the claim holds when $n = - 1$ and $k \geq 0$.) More
  generally, we wish to show that, for all $n \geq 1$,
  \begin{equation}
    (x + y)^{- n} = \sum_{k \geq 0} \binom{- n}{k}_q x^k y^{- n - k} .
    \label{eq:binomialtheorem:q:neg}
  \end{equation}
  We just found that \eqref{eq:binomialtheorem:q:neg} holds for $n = 1$. On
  the other hand, assume that \eqref{eq:binomialtheorem:q:neg} holds for some
  $n$. Then,
  \begin{eqnarray*}
    (x + y)^{- n - 1} & = & (x + y)^{- n} (x + y)^{- 1}\\
    & = & \left(\sum_{k \geq 0} \binom{- n}{k}_q x^k y^{- n - k}
    \right) \left(\sum_{k \geq 0} \binom{- 1}{k}_q x^k y^{- 1 - k}
    \right)\\
    & = & \sum_{k \geq 0} \sum_{j = 0}^k \binom{- n}{j}_q \binom{- 1}{k
    - j}_q q^{(k - j) (- n - j)} x^k y^{- n - 1 - k}\\
    & = & \sum_{k \geq 0} \binom{- n - 1}{k}_q x^k y^{- n - 1 - k},
  \end{eqnarray*}
  where the last step is an application of the generalized Chu-Vandermonde
  identity \eqref{eq:chuvandermonde:pos} with $m = - 1$. By induction,
  \eqref{eq:binomialtheorem:q:neg} therefore is true for all $n \geq 1$.
  
  We have therefore shown that \eqref{eq:binomialtheorem:q:pos} holds for all
  integers $n$. This implies the present claim in the case $k \geq 0$.
  The case when $k < 0$ can also be deduced from
  \eqref{eq:binomialtheorem:q:neg}. Indeed, observe that $x y = q^{- 1} y x$,
  so that, for any integer $n$, by \eqref{eq:binomialtheorem:q:pos} and
  \eqref{eq:binomialtheorem:q:neg},
  \begin{eqnarray*}
    (x + y)^n & = & \sum_{k \geq 0} \binom{n}{k}_{q^{- 1}} y^k x^{n -
    k}\\
    & = & \sum_{k \leq n} q^{k (n - k)} \binom{n}{k}_{q^{- 1}} x^k y^{n
    - k} = \sum_{k \leq n} \binom{n}{k}_{q^{}} x^k y^{n - k} .
  \end{eqnarray*}
  When $n \geq 0$, this is just a version of
  \eqref{eq:binomialtheorem:q:pos}. However, when $k < 0$, we deduce that
  \begin{equation*}
    \{ x^k y^{n - k} \} (x + y)^n = \binom{n}{k}_q,
  \end{equation*}
  as claimed.
\end{proof}

\section{Lucas' theorem}\label{sec:lucas}

Lucas' famous theorem \cite{lucas78} states that, if $p$ is a prime, then
\begin{equation*}
  \binom{n}{k} \equiv \binom{n_0}{k_0} \binom{n_1}{k_1} \cdots
   \binom{n_d}{k_d} \quad (\operatorname{mod} p),
\end{equation*}
where $n_i$ and $k_i$ are the $p$-adic digits of the nonnegative integers $n$
and $k$, respectively. Our first goal is to prove that this congruence in fact
holds for all integers $n$ and $k$. The next section is then concerned with
further extending these congruences to the polynomial setting.

\begin{example}
  The base $p$ expansion of a negative integer is infinite. However, only
  finitely many digits are different from $p - 1$. For instance, in base $7$,
  \begin{equation*}
    - 11 = 3 + 5 \cdot 7 + 6 \cdot 7^2 + 6 \cdot 7^3 + \ldots
  \end{equation*}
  which we will abbreviate as $- 11 = (3, 5, 6, 6, \ldots)_7$. Similarly, $-
  19 = (2, 4, 6, 6, \ldots)_7$. The extension of the Lucas congruences that is
  proved below shows that
  \begin{equation*}
    \binom{- 11}{- 19} \equiv \binom{3}{2} \binom{5}{4} \binom{6}{6}
     \binom{6}{6} \cdots = 3 \cdot 5 \equiv 1 \quad (\operatorname{mod} 7),
  \end{equation*}
  without computing that the left-hand side is $43, 758$.
\end{example}

The main result of this section, Theorem~\ref{thm:bin:lucas}, can also be
deduced from the polynomial generalization in the next section. However, we
give a direct and uniform proof here to make the ingredients more transparent.
A crucial ingredient in the usual proofs of Lucas' classical theorem is the
simple congruence
\begin{equation}
  (1 + x)^p \equiv 1 + x^p \quad (\operatorname{mod} p), \label{eq:freshman:pos}
\end{equation}
sometimes jokingly called a freshman's dream, which encodes the observation
that $\binom{p}{k}$ is divisible by the prime $p$, except in the boundary
cases $k = 0$ and $k = p$.

\begin{theorem}
  \label{thm:bin:lucas}Let $p$ be a prime. Then, for any integers $n$ and $k$,
  \begin{equation*}
    \binom{n}{k} \equiv \binom{n_0}{k_0} \binom{n'}{k'} \quad (\operatorname{mod} p),
  \end{equation*}
  where $n = n_0 + n' p$ and $k = k_0 + k' p$ with $n_0, k_0 \in \{ 0, 1,
  \ldots, p - 1 \}$.
\end{theorem}

\begin{proof}
  It is a consequence of \eqref{eq:freshman:pos} (and the algebra of Laurent
  series) that, for any prime $p$,
  \begin{equation}
    (1 + x)^{- p} \equiv (1 + x^p)^{- 1} \quad (\operatorname{mod} p),
    \label{eq:freshman:neg}
  \end{equation}
  where it is understood that both sides are expanded, as in
  \eqref{eq:f:exp:0:infty}, either around $0$ or $\infty$. Hence, in the same
  sense,
  \begin{equation}
    (1 + x)^{n p} \equiv (1 + x^p)^n \quad (\operatorname{mod} p)
    \label{eq:freshman:n}
  \end{equation}
  for any integer $n$.
  
  With the notation from the previous section, we observe that
  \begin{equation*}
    \{ x^k \} (1 + x)^n = \{ x^k \} (1 + x)^{n_0} (1 + x)^{n' p} \equiv \{
     x^k \} (1 + x)^{n_0} (1 + x^p)^{n'} \quad (\operatorname{mod} p),
  \end{equation*}
  where the congruence is a consequence of \eqref{eq:freshman:n}. Since $n_0
  \in \{ 0, 1, \ldots, p - 1 \}$, we conclude that
  \begin{equation*}
    \{ x^k \} (1 + x)^n \equiv (\{ x^{k_0} \} (1 + x)^{n_0}) (\{ x^{k' p} \}
     (1 + x^p)^{n'}) \quad (\operatorname{mod} p) .
  \end{equation*}
  This is obvious if $k \geq 0$, but remains true for negative $k$ as
  well (because $(1 + x)^{n_0}$ is a polynomial, in which case the expansions
  \eqref{eq:f:exp:0:infty} around $0$ and $\infty$ agree). Thus,
  \begin{equation*}
    \{ x^k \} (1 + x)^n \equiv (\{ x^{k_0} \} (1 + x)^{n_0}) (\{ x^{k'} \} (1
     + x)^{n'}) \quad (\operatorname{mod} p) .
  \end{equation*}
  Applying Theorem~\ref{thm:binomialtheorem:loeb} to each term, it follows
  that
  \begin{equation*}
    \binom{n}{k} \equiv \binom{n_0}{k_0} \binom{n'}{k'} \quad (\operatorname{mod} p),
  \end{equation*}
  as claimed.
\end{proof}

\section{A $q$-analog of Lucas' theorem}\label{sec:lucas:q}

Let $\Phi_m (q)$ be the $m$th cyclotomic polynomial. In this section, we prove
congruences of the type $A (q) \equiv B (q)$ modulo $\Phi_m (q)$, where $A
(q), B (q)$ are Laurent polynomials. The congruence is to be interpreted in
the natural sense that the difference $A (q) - B (q)$ is divisible by $\Phi_m
(q)$.

\begin{example}
  \label{eg:qbin:lucas}Following the notation in Theorem~\ref{thm:bin:lucas},
  in the case $(n, k) = (- 4, - 8)$, we have $(n_0, k_0) = (2, 1)$ and $(n',
  k') = (- 2, - 3)$. We reduce modulo $\Phi_3 (q) = 1 + q + q^2$. The result
  we prove below shows that
  \begin{equation*}
    \binom{- 4}{- 8}_q \equiv \binom{2}{1}_q \binom{- 2}{- 3} \quad
     (\operatorname{mod} \Phi_3 (q)) .
  \end{equation*}
  Here,
  \begin{eqnarray*}
    \binom{- 4}{- 8}_q & = & \frac{1}{q^{22}} \Phi_5 (q) \Phi_6 (q) \Phi_7
    (q)\\
    & = & \frac{1}{q^{22}} (1 - q + q^2) (1 + q + q^2 + q^3 + q^4) (1 + q +
    q^2 + \ldots + q^6)
  \end{eqnarray*}
  as well as
  \begin{equation*}
    \binom{2}{1}_q \binom{- 2}{- 3} = - 2 (1 + q),
  \end{equation*}
  and the meaning of the congruence is that
  \begin{equation*}
    \binom{- 4}{- 8}_q - \binom{2}{1}_q \binom{- 2}{- 3} = \Phi_3 (q) \cdot
     \frac{p_{21} (q)}{q^{22}},
  \end{equation*}
  where $p_{21} (q) = 1 + q^2 + 2 q^3 + q^4 + \ldots - 2 q^{19} + 2 q^{21}$ is
  a polynomial of degree $21$. Observe how, upon setting $q = 1$, we obtain
  the Lucas congruence
  \begin{equation*}
    \binom{- 4}{- 8} \equiv \binom{2}{1} \binom{- 2}{- 3} \quad (\operatorname{mod}
     3),
  \end{equation*}
  provided by Theorem~\ref{thm:bin:lucas} (the two sides of the congruence are
  equal to $35$ and $- 4$, respectively).
\end{example}

In the case $n, k \geq 0$, the following $q$-analog of Lucas' classical
binomial congruence has been obtained by Olive \cite{olive-pow} and
D\'esarm\'enien \cite{desarmenien-q}. A nice proof based on a group
action is given by Sagan \cite{sagan-qcong}, who attributes the
combinatorial idea to Strehl. We show that these congruences extend uniformly
to all integers $n$ and $k$. A minor difference to keep in mind is that the
$q$-binomial coefficients in this extended setting are Laurent polynomials
(see Example~\ref{eg:qbin:lucas}).

\begin{theorem}
  \label{thm:qbin:lucas}Let $m \geq 2$ be an integer. For any integers
  $n$ and $k$,
  \begin{equation*}
    \binom{n}{k}_q \equiv \binom{n_0}{k_0}_q \binom{n'}{k'} \quad (\operatorname{mod}
     \Phi_m (q)),
  \end{equation*}
  where $n = n_0 + n' m$ and $k = k_0 + k' m$ with $n_0, k_0 \in \{ 0, 1,
  \ldots, m - 1 \}$.
\end{theorem}

\begin{proof}
  Suppose throughout that $x$ and $y$ satisfy $y x = q x y$. It follows from
  the (noncommutative) $q$-binomial Theorem~\ref{thm:binomialtheorem:q} that,
  for nonnegative integers $m$,
  \begin{equation*}
    (x + y)^m \equiv x^m + y^m \quad (\operatorname{mod} \Phi_m (q)) .
  \end{equation*}
  As in the proof of Theorem~\ref{thm:bin:lucas} (and in the analogous sense),
  we conclude that
  \begin{equation}
    (x + y)^{n m} \equiv (x^m + y^m)^n \quad (\operatorname{mod} \Phi_m (q))
    \label{eq:freshman:n:q}
  \end{equation}
  for any integer $n$.
  
  With the notation from Section~\ref{sec:binomialtheorem}, we observe that,
  by \eqref{eq:freshman:n:q},
  \begin{equation*}
    \{ x^k y^{n - k} \} (x + y)^n \equiv \{ x^k y^{n - k} \} (x + y)^{n_0}
     (x^m + y^m)^{n'} \quad (\operatorname{mod} \Phi_m (q)) .
  \end{equation*}
  Since $n_0 \in \{ 0, 1, \ldots, p - 1 \}$, the right-hand side equals
  \begin{equation*}
    q^{(n_0 - k_0) k' m} (\{ x^{k_0} y^{n_0 - k_0} \} (x + y)^{n_0}) (\{
     x^{k' m} y^{(n' - k') m} \} (x^m + y^m)^{n'}) .
  \end{equation*}
  As $q^m \equiv 1$ modulo $\Phi_m (q)$, we conclude that $\{ x^k y^{n - k} \}
  (x + y)^n$ is congruent to
  \begin{equation*}
    (\{ x^{k_0} y^{n_0 - k_0} \} (x + y)^{n_0}) (\{ x^{k' m} y^{(n' - k') m}
     \} (x^m + y^m)^{n'})
  \end{equation*}
  modulo $\Phi_m (q)$. Observe that the variables $X = x^m$ and $Y = y^m$
  satisfy the commutation relation $Y X = q^{m^2} X Y$. Hence, applying
  Theorem~\ref{thm:binomialtheorem:loeb:q} to each term, we conclude that
  \begin{equation*}
    \binom{n}{k}_q \equiv \binom{n_0}{k_0}_q \binom{n'}{k'}_{q^{m^2}} \quad
     (\operatorname{mod} \Phi_m (q)) .
  \end{equation*}
  Since $q^{m^2} \equiv 1$ modulo $\Phi_m (q)$, the claim follows.
\end{proof}

\section{Conclusion}

We believe (and hope that the results of this paper provide some evidence to
that effect) that the binomial and $q$-binomial coefficients with negative
entries are natural and beautiful objects. On the other hand, let us indicate
an application, taken from \cite{s-apery}, of binomial coefficients with
negative entries.

\begin{example}
  A crucial ingredient in Ap\'ery's proof \cite{apery} of the
  irrationality of $\zeta (3)$ is played by the {\emph{Ap\'ery numbers}}
  \begin{equation}
    A (n) = \sum_{k = 0}^n \binom{n}{k}^2 \binom{n + k}{k}^2 .
    \label{eq:apery}
  \end{equation}
  These numbers have many interesting properties. For instance, they satisfy
  remarkably strong congruences, including
  \begin{equation}
    A (p^r m - 1) \equiv A (p^{r - 1} m - 1) \quad (\operatorname{mod} p^{3 r}),
    \label{eq:apery-sc1}
  \end{equation}
  established by Beukers \cite{beukers-apery85}, and
  \begin{equation}
    A (p^r m) \equiv A (p^{r - 1} m) \quad (\operatorname{mod} p^{3 r}),
    \label{eq:apery-sc}
  \end{equation}
  proved by Coster \cite{coster-sc}. Both congruences hold for all primes $p
  \geq 5$ and positive integers $m, r$. The definition of the Ap\'ery
  numbers $A (n)$ can be extended to all integers $n$ by setting
  \begin{equation}
    A (n) = \sum_{k \in \mathbb{Z}} \binom{n}{k}^2 \binom{n + k}{k}^2,
    \label{eq:apery:neg}
  \end{equation}
  where the binomial coefficients are now allowed to have negative entries.
  Applying the reflection rule \eqref{eq:bin:neg} to \eqref{eq:apery:neg}, we
  obtain
  \begin{equation}
    A (- n) = A (n - 1) . \label{eq:A:neg}
  \end{equation}
  In particular, we find that the congruence \eqref{eq:apery-sc1} is
  equivalent to \eqref{eq:apery-sc} with $m$ replaced with $- m$. By working
  with binomial coefficients with negative entries, the second author gave a
  uniform proof of both sets of congruences in \cite{s-apery}. In addition,
  the symmetry \eqref{eq:A:neg}, which becomes visible when allowing negative
  indices, explains why other Ap\'ery-like numbers satisfy
  \eqref{eq:apery-sc} but not \eqref{eq:apery-sc1}.
\end{example}

We illustrated that the Gaussian binomial coefficients can be usefully
extended to the case of negative arguments. More general binomial
coefficients, formed from an arbitrary sequence of integers, are considered,
for instance, in \cite{kw-xbin} and it is shown by Hu and Sun
\cite{hs-lucas} that Lucas' theorem can be generalized to these. It would be
interesting to investigate the extent to which these coefficients and their
properties can be extended to the case of negative arguments. Similarly, an
elliptic analog of the binomial coefficients has recently been introduced by
Schlosser \cite{schlosser-bin}, who further obtains a general noncommutative
binomial theorem of which Theorem~\ref{thm:binomialtheorem:q} is a special
case. It is natural to wonder whether these binomial coefficients have a
natural extension to negative arguments as well.

In the last section, we showed that the generalized $q$-binomial coefficients
satisfy Lucas congruences in a uniform fashion. It would be of interest to
determine whether other well-known congruences for the $q$-binomial
coefficients, such as those considered in \cite{andrews-qcong99} or
\cite{straub-qljunggren}, have similarly uniform extensions.

\begin{acknowledgements}
Part of this work was completed while the first author was supported by a
Summer Undergraduate Research Fellowship (SURF) through the Office of
Undergraduate Research (OUR) at the University of South Alabama.
We are grateful to Wadim Zudilin for helpful comments on an earlier draft of
this paper.
\end{acknowledgements}

\end{document}